\documentclass[12pt,reqno]{amsart}
\usepackage{amsfonts,amssymb,hyperref,amsthm,enumerate,color,stmaryrd,pgf,tikz,comment}
\usepackage[left=2.6cm,right=2.6cm,top=3cm,bottom=3cm,bindingoffset=0cm]{geometry}
\newtheorem{thm}{Theorem}[section]
\newtheorem{lem}[thm]{Lemma}
\newtheorem{prop}[thm]{Proposition}
\newtheorem{cor}[thm]{Corollary}
\newtheorem{hypo}[thm]{Hypothesis}

\theoremstyle{remark}
\newtheorem*{rem}{Remark}

\theoremstyle{definition}

\newtheorem{defi}[thm]{Definition}
\def\F{\mathcal{F}}
\def\B{\mathcal{B}}

\def\G{\Gamma}

\def\la{\langle}
\def\ra{\rangle}

\def\Z{\mathbb{Z}}
\DeclareMathOperator{\AGL}{AGL}
\DeclareMathOperator{\aut}{Aut}

\DeclareMathOperator{\core}{core}
\DeclareMathOperator{\GL}{GL}
\DeclareMathOperator{\GaL}{\Gamma L}
\DeclareMathOperator{\nest}{\mathcal{N}}
\DeclareMathOperator{\orb}{Orb}
\DeclareMathOperator{\PG}{PG} 
\DeclareMathOperator{\PGL}{PGL}
\DeclareMathOperator{\PGaL}{P\Gamma L}
\DeclareMathOperator{\PSL}{PSL}
\DeclareMathOperator{\SiL}{\Sigma L}
\DeclareMathOperator{\soc}{soc}
\DeclareMathOperator{\sym}{Sym}
\begin{document}
\title[Edge-transitive core-free Nest graphs]
{Edge-transitive core-free Nest graphs}
\author[I.~Kov\'acs]{Istv\'an~Kov\'acs}
\address{I.~Kov\'acs 
\newline\indent
UP IAM, University of Primorska, Muzejski trg 2, SI-6000 Koper, Slovenia 
\newline\indent
UP FAMNIT, University of Primorska, Glagol\v jaska ulica 8, SI-6000 Koper, Slovenia}
\email{istvan.kovacs@upr.si}
\thanks{Partially supported by the Slovenian Research Agency (research program P1-0285, research projects N1-0062, J1-9108, J1-1695, J1-2451, N1-0140 and N1-0208).}
\maketitle
\begin{abstract}
A finite simple graph $\G$ is called a Nest graph if it is regular of valency $6$ and admits an automorphism $\rho$ with two orbits of the same length such that at least one of the subgraphs induced by these orbits is a cycle. 
We say that $\G$ is core-free if no non-trivial subgroup of the group generated by 
$\rho$ is normal in $\aut(\G)$. In this paper we show that, if $\G$ is 
edge-transitive and core-free, then it is isomorphic to 
one of the following graphs: the complement of the Petersen graph, the Hamming graph 
$H(2,4)$, the Shrikhande graph and a certain normal $2$-cover of $K_{3,3}$ by 
$\Z_2^4$.
\end{abstract}
\section{Introduction}\label{sec:intro}
All groups in this paper will be finite and all graphs will be finite and simple.
A graph admitting an automorphism with two orbits of the same length is called 
a {\em bicirculant}. Symmetry properties of bicirculants have attracted 
considerable attention (see, e.g., \cite{AHK,CZFZ,DGJ,KKMW,MMSF,MP,P,ZZ}). 
Following~\cite{KR}, for an integer $d \ge 3$, 
we denote by $\F(d)$ the family of regular graphs having
valency $d$ and admitting an automorphism with two orbits of the same length such that at least one of the subgraphs induces by these orbits is a cycle.  
Jajcay et al.~\cite{JMSV} initiated the investigation of the edge-transitive graphs 
in the classes $\F(d)$, $d \ge 6$. 
The families $\F(d)$ with $3 \le d \le 5$ were studied under different names.   
The graphs in $\F(3)$  
were introduced by Watkins~\cite{Wa} under the 
name {\em generalised Petersen graphs}, the graphs in $\F(4)$ by 
Wilson~\cite{Wi} under the name {\em Rose Window graphs}, and the 
graphs in $\F(5)$ by Arroyo et al.~\cite{AHKOS} under the name  
{\em Taba\v{c}jn graphs}. The automorphism groups of these graphs form the subject 
of the papers \cite{FGW,KKM,DKM,KMMS}, and the question which of them are edge-transitive has been answered in \cite{FGW,KKM,AHKOS}. 

Jajcay et al.~\cite{JMSV} asked whether there exist edge-transitive 
graphs in $\F(d)$ for $d \ge 6$. Following \cite{V}, they call the graphs in 
$\F(6)$ {\em Nest graphs}. Several infinite families of edge-transitive Nest graphs 
were exhibited, which turn out to have interesting properties 
(e.g., half-arc-transitivity). However, no edge-transitive graph of valency larger than
$6$ was found. Recently, it was proved by the author and Ruff~\cite{KR} that 
the complement of the Petersen graph is the only edge-transitive graph in 
$\F(d)$ with $d \ge 6$, which has twice an odd number of vertices.  
The main result of \cite{JMSV} is the classification of 
the edge-transitive Nest graphs of girth $3$ (see \cite[Theorem~8]{JMSV}), and  
the task to classify all edge-transitive Nest graphs was posed as  \cite[Problem~2]{JMSV}. In what follows, the Nest graphs will be described via their representation due to \cite[Construction~3]{JMSV}, which goes as follows. 
Let $n \ge 4$ and let $a, b, c, k \in \Z_n$ such that each of them is distinct 
from $0$ (the zero element of $\Z_n$), the elements $a, b$ and $c$ are pairwise distinct, and in the case when $n$ is even, $k \ne n/2$.  
Then the {\em Nest graph} $\nest(n;a,b,c;k)$ is defined to have vertex set    
$\{ u_i : i \in \Z_n\} \cup \{ v_i : i \in \Z_n\}$, and three types of edges such as
\begin{itemize}
\setlength{\itemsep}{0.25\baselineskip}
\item $\{u_i,u_{i+1}\}$ for $i \in \Z_n$  ({\em rim edges}),
\item $\{v_i,v_{i+k}\}$ for $i \in \Z_n$   ({\em hub edges}),
\item $\{u_i,v_i\}$,  $\{u_i,v_{i+a}\}$, $\{u_i,v_{i+b}\}$ and 
$\{u_i,v_{i+c}\}$ for $i \in \Z_n$   ({\em spoke edges}),
\end{itemize}
where the sums in the subscripts are computed in $\Z_n$. 
It is easy to see that the permutation $\rho$ of $V(\G)$, defined as 
$\rho=(u_0,u_1,\ldots,u_{n-1})(v_0,v_1,\ldots,v_{n-1})$, 
is an automorphism of $\G$ with orbits $\{u_i : i \in \Z_n\}$ and 
$\{v_i : i \in \Z_n\}$, and the subgraph induced 
by the former orbit is a cycle. It is not hard to show that all the graphs 
$\nest(n;a,b,c;k)$ comprise the whole family $\F(6)$. 
\medskip

In the case of both the Rose Window and the 
Taba\v{c}jn graphs, the classification of the edge-transitive graphs was obtained in two 
main steps. The so called core-free graphs were found first and the  
rest was retrieved from the core-free graphs using 
covering techniques (see \cite{KKM,AHKOS}).  
Here is the formal definition of a core-free Nest graph.

\begin{defi}\label{core-free}
Let $\G=\nest(n;a,b,c;k)$ be a Nest graph and $\rho$ be the permutation of $V(\G)$ defined as  $\rho=(u_0,u_1,\ldots,u_{n-1})(v_0,v_1,\ldots,v_{n-1})$. 
Then $\G$ is {\em core-free} if no non-trivial 
subgroup of $\la \rho \ra$ (the group generated by $\rho$) is normal in 
$\aut(\G)$. 
\end{defi}

\begin{rem}
The term ``core-free'' comes from group theory. 
For a subgroup $A \le B$, the {\em core} of $A$ in 
$B$ is the largest normal subgroup of $B$ contained in $A$. In the case when $A$ 
has trivial core, it is also called {\em core-free}.  In this context, Definition~\ref{core-free}  can be rephrased by saying that $\G$ is core-free if and only if $\la \rho \ra$ is core-free 
in $\aut(\G)$.   
\end{rem}

Our goal in this paper is to determine the edge-transitive core-free Nest graphs. 
For an explanation why this task is a more subtle 
than in the case of Rose Window and Taba\v{c}jn graphs, we refer to~\cite[p.~9]{JMSV}. The edge-transitive non-core-free Nest graphs are handled in the paper~\cite{K22}. 
\medskip

The main result of this paper is the following theorem.

\begin{thm}\label{main}
If $\nest(n;a,b,c;k)$ is an edge-transitive core-free graph, then it is isomorphic to 
one of the following graphs: 
$$
\nest(5;1,2,3;2),~\nest(8;1,3,4;3),~\nest(8;1,2,5;3)~\text{and}~\nest(12;2,4,8;5).
$$
\end{thm}

\begin{rem}
The fact that each of the Nest graphs in Theorem~\ref{main} 
is core-free was mentioned by Jajcay et al., see \cite[p.~9]{JMSV}.
The first three of them are well-known strongly regular graphs. The Nest graph 
$\nest(5;1,2,3;2)$ is the complement of the {\em Petersen graph}, 
$\nest(8;1,3,4;3)$ is the {\em Hamming graph} $H(2,4)$, and 
$\nest(8;1,2,5;3)$ is the {\em Shrikhande graph}.  
The fourth Nest graph $\nest(12;2,4,8;5)$ is not strongly-regular, it can be described as a normal $2$-cover of the complete bipartite graph $K_{3,3}$ 
by $\Z_2^4$ (for the definition of a normal $2$-cover, see the 2nd paragraph of 
Subsection~2.1).  
\end{rem}

The paper is organised as follows. 
Section~\ref{sec:known} contains the needed results from graph and group theory. 
In Section~\ref{sec:Nest} we review some results about Nest graphs obtained in 
\cite{JMSV,KR}. Section~\ref{sec:family} is devoted to the Nest graphs in the form 
$\nest(2m;2,m,2+m;1)$, $m$ is odd. The main result (Proposition~\ref{family}) 
is a characterisation, which was mentioned in \cite{JMSV} 
without a proof, and which is needed for us in the proof Theorem~\ref{main}. 
The latter proof is presented in Section~\ref{sec:proof}.
\section{Preliminaries}\label{sec:known}
\subsection{Graph theory}

Given a graph $\G$, let $V(\G)$, $E(\G)$, $A(\G)$ and $\aut(\G)$ denote   
its {\em vertex set}, {\em edge set}, {\em arc set} and {\em automorphism group}, respectively. The number $|V(\G)|$ is called the {\em order} of $\G$. 
The set of vertices adjacent with a given vertex $v$ is denoted by $\G(v)$. 
If $G \leq \aut(\G)$ and $v \in V(\G)$, then the {\em stabiliser} of $v$ in 
$G$ is denoted by $G_v$, the {\em orbit} of $v$ under $G$ by $v^G$, and 
the set of all $G$-orbits by $\orb(G,V(\G))$.  
If $B \subseteq V(\G)$, then the {\em setwise stabiliser} of $B$ in $G$ is denoted 
by $G_{\{B\}}$. 
If $G$ is transitive on $V(\G)$, then $\G$ is said to be {\em $G$-vertex-transitive},     
and $\G$ is simply called {\em vertex-transitive} when it is $\aut(\G)$-vertex-transitive;  ({\em $G$-}){\em edge-} and ({\em $G$-}){\em arc-transitive} graphs are defined correspondingly. 
\medskip

Let $\pi$ be an arbitrary partition of $V(\G)$ and for a vertex $v \in V(\G)$, 
let $\pi(v)$ denote the class containing $v$. 
The {\em quotient graph} of $\G$ with respect to $\pi$, denoted by $\G/\pi$, is 
defined to have vertex set $\pi$, and edges 
$\{ \pi(u), \pi(v) \}$, where $\{u,v\} \in E(\G)$ such that $\pi(u) \ne \pi(v)$. 
Now, if there exists a constant $r$ such that 
$$
\forall \{u,v\} \in E(\G):~ \pi(u) \ne \pi(v)~\text{and}~|\G(u) \cap \pi(v)|=r,
$$
then $\G$ is called an {\em $r$-cover} of $\G/\pi$. The term {\em cover} will also 
be used instead of $1$-cover. 
In the special case when $\pi=\orb(N,V(\G))$ for an intransitive normal subgroup $N \lhd \aut(\G)$, $\G/N$ will also be written for 
$\G/\pi$ and when $\G$ is also an $r$-cover (cover, respectively) of $\G/N$, then 
the term {\em normal $r$-cover} ({\em normal cover}, respectively) will also be used.  
It is well-known that this is always the case when $\G$ is edge-transitive. 
More precisely, if $\G$ is a $G$-edge-transitive graph, $\G$ is regular with  
valency $\kappa$, and $N \lhd G$ is intransitive, then 
$\G$ is a normal $r$-cover of $\G/N$ for some $r$ and $r$ divides $\kappa$.
\medskip

A graph admitting a regular cyclic group of automorphisms is called {\em circulant}. 
A recursive classification of finite arc-transitive circulants was obtained independently by Kov\'acs~\cite{K05} and Li~\cite{Li}. The paper~\cite{K05} also provides an explicit 
characterisation~(see \cite[Theorem~4]{K05}), which was rediscovered recently by 
Li et al.~\cite{LXZ}. The characterisation presented below follows from the proof of 
\cite[Theorem~4]{K05} or from \cite[Theorem~1.1]{LXZ}. 
\medskip

In what follows, given a cyclic group $C$ and a divisor 
$d$ of $|C|$, we denote by $C_d$ the unique subgroup of $C$ of order $d$.  

\begin{thm}\label{K}
{\rm (\cite{K05})} 
Let $\G$ be a connected arc-transitive graph of order $n$ and of valency $\kappa$ 
and suppose that $C \le \aut(\G)$ 
is a regular cyclic subgroup. Then one of the following holds.
\begin{enumerate}[{\rm (a)}]
\item $\G$ is the complete graph.
\item $C$ is normal in $\aut(\G)$.
\item $\B=\orb(C_d,V(\G))$ is a block system 
for $\aut(\G)$ for some divisor $d$ of $\gcd(n,\kappa)$, $d > 1$.   
$\G$ is a normal $d$-cover of $\G/\B$, and $\G/\B$ is a connected  
arc-transitive circulant of valence $\kappa/d$. 
\item $\B_1=\orb(C_d,V(\G))$ and $\B_2=\orb(C_{n/d},V(\G))$ are 
block systems for $\aut(\G)$ for some divisor $d$ of $n$ such that $d > 3$, 
$\gcd(d,n/d)=1$ and $d-1$ divides $\kappa$. $\G/\B_1$ is a connected 
arc-transitive circulant of valency $\kappa/(d-1)$, $\G/\B_2 \cong K_d$, and 
\begin{equation}\label{eq:aut}
\aut(\G)=G_1 \times G_2,
\end{equation}
where $C_d \le G_1$, $G_1 \cong S_d$, $C_{n/d} < G_2$, 
and $G_2 \cong \aut(\G/\B_1)$.
\end{enumerate}
\end{thm}

\begin{rem}
Although not used later, it is worth mentioning that the graph $\G$ in part (c) is isomorphic to the lexicographical product $\G/\B[\overline{K}_d]$, 
where $\overline{K}_d$ is the edgeless 
graph on $d$ vertices, and the graph $\G$ in part (d) is isomorphic to the 
tensor (direct) product $K_d \times \G/\B_1$ (see \cite{K05,LXZ}). 
\end{rem}

In the rest of the section we restrict ourselves to arc-transitive circulants of small 
valency. 

\begin{lem}\label{BFSX}
{\rm (\cite[part (ii) of Corollary~1.3]{BFSX})}
Let $\G$ be a connected arc-transitive graph of order $n$ and of 
valency $\kappa$, where $\kappa=3$ or $4$, and suppose that $C \le \aut(\G)$ 
is a regular cyclic subgroup.
Then one of the following holds.
\begin{enumerate}[{\rm (1)}] 
\item $\G$ is isomorphic to one of the graphs: $K_4, K_5, K_{3,3}$ and 
$K_{5,5}-5K_2$.
\item $\kappa=4$ and $C$ is normal in $\aut(\G)$.
\item $\kappa=4$, $n$ is even, $\B=\orb(C_2,V(\G))$ is a block system for $\aut(\G)$, 
and $\G$ is a normal $2$-cover of $\G/\B$, which is a cycle.
\end{enumerate}
\end{lem}

\begin{lem}\label{6}
Let $\G$ be a connected arc-transitive graph of order $n > 14$ and of valency $6$, and suppose that $C \le \aut(\G)$ is a regular cyclic subgroup.
Then $\aut(\G)$ contains a normal subgroup $N$ such that one of the following 
holds.  
\begin{enumerate}[{\rm (1)}]
\item $N=C$, or 
\item $n \equiv 4 \!\!\pmod 8$ and $N=C_{n/4}$, or   
\item $N \cong \Z_3^\ell$ for $\ell \ge 2$ and $C_3 < N$.
\end{enumerate}
\end{lem}
\begin{proof} 
$\G$ belongs to one of the families (a)--(d) in Theorem~\ref{K}. 

Family~(a): This case cannot occur as $n > 14$.

Family~(b): Part (1) follows.

Family~(c): In this case $\orb(C_d,V(\G))$ is a block system for $\aut(\G)$, where 
$d=2$ or $d=3$. Let $\B=\orb(C_d,V(\G))$. 

If $d=2$, then $\G/\B$ has valency $3$. It follows from 
Lemma~\ref{BFSX} that $n \le 12$, but this is excluded. 

If $d=3$, then choose $N$ to be the 
Sylow $3$-subgroup of the kernel of the action of $\aut(\G)$ on $\B$. 
It follows that $N \cong \Z_3^\ell$ for some $\ell \ge 2$. Also, $N$ is characteristic in the 
latter kernel, which implies that $N \lhd \aut(\G)$. Finally, 
$\orb(N,V(\G))=\B=\orb(C_3,V(\G))$, and so $C_3 < N$, i.e., part (3) holds.  

Family~(d): In this case it follows from the assumption that $n > 14$ that     
$\orb(C_4,V(\G))$ and $\orb(C_{n/4},V(\G))$ are block systems for 
$\aut(\G)$ and $n \equiv 4 \!\!\pmod 8$.
Furthermore, 
$$
\aut(\G) = G_1 \times G_2,
$$
where $C_4 < G_1$, $G_1 \cong S_4$, $C_{n/4} < G_2$ and 
$G_2 \cong  \aut(\G/\B_1)$, where $\B_1=\orb(C_4,V(\G))$. 
The graph $\G/\B_1$ is connected of valency $2$, hence it is a 
cycle of length $n/4$. It follows that $C_{n/4}$ is characteristic in 
$G_2$, and as $G_2 \lhd \aut(\G)$, part (2) follows. 
\end{proof}
\subsection{Group theory}

Our terminology and notation are standard and 
we follow the books~\cite{DM,H}. 
The {\em socle} of a group $G$, denoted by $\soc(G)$, is the subgroup generated by the set of all minimal normal subgroups (see~\cite[p.~111]{DM}). 
The group $G$ is called {\em almost simple} if $\soc(G)=T$, where $T$ is a non-abelian simple group. In this case $G$ is embedded in $\aut(T)$ so that its socle is embedded 
via the inner automorphisms of $T$, and we also write $T \le G \le \aut(T)$ (see \cite[p.~126]{DM}).  

Our proof of Theorem~\ref{main} relies on 
the classification of primitive groups containing a cyclic subgroup with two orbits due to 
M\"uller~\cite{M}.  Here we need only the special case when the cyclic subgroup is semiregular. 

\begin{thm}\label{M}
{\rm (\cite[Theorem~3.3]{M})} 
Let $G$ be a primitive permutation group of degree $2n$ containing an element with two orbits of the same length. Then one of the following holds, where $G_0$ denotes the 
stabiliser of a point in $G$. 
\begin{enumerate}[{\rm (1)}] 
\item (Affine action) $\Z_2^m \lhd G \le \AGL(m,2)$, 
where $n=2^{m-1}$. Furthermore, one of the following holds.
\begin{enumerate}[{\rm (a)}] 
\item $n=2$, and $G_0=\GL(2,2)$.
\item $n=2$, and $G_0=\GL(1,4)$.
\item $n=4$, and $G_0=\GL(3,2)$.
\item $n=8$, and $G_0$ is one of the following groups: $\Z_5 \rtimes \Z_4$, 
$\GaL(1,16)$, $(\Z_3 \times \Z_3) \rtimes \Z_4$, $\SiL(2,4)$, $\GaL(2,4)$, 
$A_6$, $\GL(4,2)$, $(S_3 \times S_3) \rtimes \Z_2$, $S_5$, $S_6$ and $A_7$.
\end{enumerate}
\item 
(Almost simple action) $G$ is an almost simple group and one of the following holds. 
\begin{enumerate}[{\rm (a)}]
\item $n \ge 3$, $\soc(G)=A_{2n}$, and 
$A_{2n} \le G \le S_{2n}$ in its natural action. 
\item $n=5$, $\soc(G)=A_5$, and $A_5 \le G \le S_5$ in its action on the set of $2$-subsets of $\{1,2,3,4,5\}$.
\item $n=(q^d-1)/2(q-1)$, $\soc(G)=\PSL(d,q)$, and $\PSL(d,q) \le G \le 
\PGaL(d,q)$ for some odd prime power $q$ and even number $d  \ge 2$ such that $(d,q) \ne (2,3)$.  
\item $n=6$ and $\soc(G)=G=M_{12}$.
\item $n=11$, $\soc(G)=M_{22}$, and $M_{22} \le G \le \aut(M_{22})$.
\item $n=12$ and $\soc(G)=G=M_{24}$.
\end{enumerate}
\end{enumerate}
\end{thm}

If $G$ is a group in one of the families (a)-(f) in part (2) above, then it follows from 
\cite[Theorem~4.3B]{DM} that $\soc(G)$ is the unique 
minimal normal subgroup of $G$. Therefore, we have the following corollary.
 
\begin{cor}\label{M-p1}
Let $G$ be a primitive permutation group in one of the families (a)-(f) in part (2) of 
Theorem~\ref{M}, and let $N \lhd G$, $N \ne 1$. Then $N$ is also primitive.
\end{cor}

For a transitive permutation group $G \le \sym(\Omega)$, the {\em subdegrees} of 
$G$ are the lengths of the orbits of a point stabiliser $G_\omega$, 
$\omega \in \Omega$. 
Since $G$ is transitive, it follows that the subdegrees do not depend on the choice of 
$\omega$ (see~\cite[p.~72]{DM}). The number of orbits of $G_\omega$ is called 
the {\em rank} of $G$. 
The actions of a group $G$ on sets $\Omega$ and $\Omega'$ are said to be 
{\em equivalent} if there is a bijection $\varphi : \Omega \to \Omega'$ 
such that 
$$
\forall \omega \in \Omega,~\forall g \in G:~
\varphi(\omega^g)=(\varphi(\omega))^g.
$$ 

Now, suppose that $G$ is a group in one of the 
families (a)-(f) in part (2) of Theorem~\ref{M}.
If $G$ is in family (a), then the action is unique up to equivalence and $G$ is clearly 
$2$-transitive. If $G$ is in family (b), then the action is unique up to equivalence and the subdegrees are $1, 3$ and $6$. 
Let $G$ be in family (c). The semiregular cyclic subgroup of $G$ with two orbits is 
contained in a regular cyclic group, called the 
{\em Singer subgroup} of $\PGL(d,q)$ (see \cite[Chapter~2, Theorem~7.3]{H}). 
In this case the action is unique up to equivalence if and only if $d=2$. 
If $d \ge 4$, then the action of $G$ is equivalent to either its natural action on the set of 
points of the projective geometry $\PG(d-1,q)$, or to its natural action on the set of hyperplanes of $\PG(d-1,q)$. In both actions $G$ is $2$-transitive. 
Finally, if $G$ is in the families (d)-(f), then the action is unique up to equivalence  
and $G$ is $2$-transitive (this can also be read off from~\cite{Atlas}). 
All this information is summarised in the lemma below. 

\begin{lem}\label{M-p2&3}
Let $G$ be a primitive permutation group in one of the families 
(a)-(f) in part (2) of Theorem~\ref{M}.    
\begin{enumerate}[{\rm (1)}]
\item $G$ is $2$-transitive, unless $G$ belongs to family (b). In the latter case the  subdegrees are $1, 3$ and $6$. 
\item The action of $G$ is unique up to equivalence, unless $G$ is in family 
(c) and $d \ge 4$. In the latter case $G$ admits two inequivalent 
faithful actions, namely, the natural actions on the set of points and the set of 
hyperplanes, respectively, of the projective geometry $\PG(d-1,q)$.
\end{enumerate}
\end{lem}

The following result about $G$-arc-transitive bicirculants can be found in 
Devillers et al.~\cite{DGJ}. The proof works also for the edge-transitive 
bicirculants, in fact, it is an easy consequence of Theorem~\ref{M}. 

\begin{prop}\label{DGJ}
{\rm (\cite[part (1) of Proposition~4.2]{DGJ})}
Let $\G$ be a $G$-edge-transitive bicirculant such that $G$ is a primitive 
group. Then $\G$ is one of the following graphs:
\begin{enumerate}[{\rm (1)}] 
\item The complete graph, and $G$ is one of the $2$-transitive groups 
described in part (2) of Theorem~\ref{M}.
\item The Petersen graph or its complement, and $A_5 \le G \le S_5$.
\item The Hamming graph $H(2,4)$ or its complement, and $G$ is a rank $3$ subgroup of 
$\AGL(4,2)$.
\item The Clebsch graph or its complement, and $G$ is a rank $3$ subgroup of 
$\AGL(4,2)$. 
\end{enumerate}
\end{prop}

Using the computational result that there exists no edge-transitive graph in $\F(d)$ 
with $7 \le d \le 10$ and of order at most $100$, one can easily deduce 
which of the graphs in the families (1)--(4) above 
belongs also to the family $\F(d)$ for some $d \ge 3$. 

\begin{cor}\label{cor-DGJ}
Let $\G \in \F(d)$ be a $G$-edge-transitive graph for some $d \ge 3$. 
If $G$ is primitive on $V(\G)$, then $\G$ is isomorphic to one of the graphs: 
$K_6$, the Petersen graph and its complement, and the Hamming graph $H(2,4)$.  
\end{cor}

Finally, we also need a result of Lucchini~\cite{Lu} about core-free cyclic subgroups 
(this serves as a key tool in \cite{KKM,AHKOS} as well). 
For the definition of a core-free subgroup, see the remark following 
Definition~\ref{core-free}. 

\begin{thm}\label{L}
{\rm (\cite{Lu})}
If $C$ is a core-free cyclic proper subgroup of a group $G$, then 
$|C|^2 < |G|$.
\end{thm}
\section{Nest graphs}\label{sec:Nest}
In this section we review some previous results about Nest graphs, which were obtained in \cite{JMSV,KR}. 

\begin{lem}\label{JMSV1}
{\rm (\cite[Lemma~4]{JMSV})}
Let $\G=\nest(n;a,b,c;k)$ and suppose that $c=a+b$ (in $\Z_n$). 
Then $\G$ is edge-transitive if and only if it is also 
arc-transitive.
\end{lem}

The next result establishes some obvious isomorphisms.  

\begin{lem}\label{JMSV2}
{\rm (\cite[Lemma~5]{JMSV})}
The graph $\nest(n;a,b,c;k)$ is isomorphic to  
$\nest(n;a',b',c';k)$, where $\{a,b,c\}=\{a',b',c'\}$, as well as to any of 
the graphs:  
$$
\nest(n;a,b,c;-k),~\nest(n;-a,-b,-c;k)~\text{and}~\nest(n;-a,b-a,c-a;k).
$$ 
\end{lem}

The graphs in the next lemma will be further studied in the next section.

\begin{lem}\label{JMSV3}
{\rm (\cite[Lemma~6]{JMSV})}
If $m \ge 3$ is an odd integer, then the graph 
$\nest(2m;2,m,2+m;1)$ is arc-transitive having vertex stabilisers of order $12$.
Furthermore, the stabiliser of $u_0$ in $\aut(\G)$ is the dihedral group $D_{6}$ of 
order $12$ generated by the involutions $\varphi$ and $\eta$ defined by 
$$
u_i^{\varphi}=\begin{cases} u_{-i} & \text{if $i$ is even}, \\ 
v_{-i+1} & \text{if $i$ is odd}\end{cases}\quad \text{and} \quad 
v_i^{\varphi}=\begin{cases} u_{-i+1} & \text{if $i$ is even}, \\ 
v_{-i+2} & \text{if $i$ is odd},\end{cases}
$$
and $u_i^\eta=u_i$ and $v_i^\eta=v_{i+m}$ for every $i \in \Z_n$.
\end{lem}

Suppose that $\G$ is a $G$-edge-transitive Nest graph. 
In the next two lemmas we consider block systems for $G$. 
A block system $\B$ is said to be {\em minimal} 
if it is non-trivial, and no non-trivial block for $G$ is contained properly in a block 
of $\B$ (by {\em non-trivial} we mean that the block is neither a singleton 
subset nor the whole vertex set). We say that $\B$ is {\em normal} if 
$\B=\orb(N,V(\G))$ for some $N \lhd G$. Furthermore, we say that $\B$ is {\em cyclic} if any 
block in $\B$ is contained in either $\{ u_i : i \in \Z_n\}$ or 
$\{ v_i : i \in \Z_n\}$.  

\begin{lem}\label{KR1}
{\rm (\cite[Lemma~4.1]{KR})} 
Let $\G$ be a $G$-edge-transitive Nest graph of order $2n$ such that 
$$
C < G,~C=\la \rho \ra,~\text{and}~  
\rho=(u_0,u_1,\ldots,u_{n-1})(v_0,v_1,\ldots,v_{n-1}),
$$ 
and let $\B$ be a cyclic 
block system for $G$ with blocks of size $d$, $d < n/2$.  
Then the following hold.
\begin{enumerate}[{\rm (1)}]
\item The kernel of the action of $G$ on $\B$ is equal to $C_d$ (the subgroup of $C$ of order $d$). 
\item $\G$ is a normal cover of $\G/\B$.
\item $\G/\B$ is a $\bar{G}$-edge-transitive Nest graph of order $2n/d$, where 
$\bar{G}$ is the image of $G$ induced by its action on $\B$. 
\end{enumerate}
\end{lem}

\begin{rem}
Suppose that the graph $\G$ in the lemma above is given as 
$\G=\nest(n;a,b,c;k)$ for $a, b, c, k \in \Z_n$. 
Note that, then 
$\G/\B \cong \nest(n/d;f(a),f(b),f(c);f(k))$, where 
$f$ is the homomorphism from $\Z_n$ to $\Z_{n/d}$ such that $f(1)=1$. 
\end{rem}

\begin{lem}\label{KR2}
{\rm (\cite[Lemma~4.2]{KR})} 
Let $\G$ be a $G$-edge-transitive Nest graph of order $2n$ such that 
$$
C < G,~C=\la \rho \ra,~\text{and}~  
\rho=(u_0,u_1,\ldots,u_{n-1})(v_0,v_1,\ldots,v_{n-1}),
$$ 
and let $\B$ be a non-cyclic block system for $G$ with blocks of size $d$.
Then the following hold.
\begin{enumerate}[{\rm (1)}]
\item The number $d$ is even and any block in $\B$ is a union of two $C_{d/2}$-orbits. 
\item The group $C$ acts transitively on $\B$ and the kernel of the action of $C$ 
on $\B$ is equal to $C_{d/2}$. 
\item If $d > 2$ and $\B$ is minimal, then $\B$ is normal.
\end{enumerate}
\end{lem}
\section{A property of the graphs $\nest(2m;2,m,2+m;1)$, $m$ is odd}\label{sec:family}
In this section we give the following characterisation of the Nest graphs in the title.
As we said in the introduction, this was mentioned already in \cite{JMSV} 
without a proof.

\begin{prop}\label{family}
Let $\G=\nest(n;a,b,c;k)$ be an edge-transitive graph such that $n > 8$ and 
suppose that there exists a non-identity automorphism of $\G$, which fixes all vertices 
$u_i$, $i\in \Z_n$. Then $\G \cong \nest(2m;2,m,2+m;1)$ for some odd 
number $m$. 
\end{prop}

We prove first an auxiliary lemma.

\begin{lem}\label{help}
Let $\G=\nest(2m;a,m,a+m;k)$, where $m > 2$, $a=2$ or $m-2$, and 
$k=1$ or $m-1$. Then $\G$ is edge-transitive if and only if $m$ is odd and 
$\G \cong \nest(2m;2,m,2+m;1)$.
\end{lem}
\begin{proof}
The ``if'' part follows from Lemma~\ref{JMSV3}.  

For the ``only if'' part, assume that $\G$ is edge-transitive. 
Applying Lemma~\ref{JMSV2} to $\G$, we find that 
$$
\G  \cong \G':=\nest(2m;2,m,2+m;k),
$$ 
where $k=1$ or $m-1$. We have to show that $m$ is odd and $k=1$.

Assume on the contrary that $m$ is even or $k=m-1$. 
Moreover, let us choose $m$ so that it is the smallest number for which this happens. 
A quick check with the computer algebra package 
{\sc Magma}~\cite{BCP} shows that $m \ge  12$.

Define the binary relation $\sim$ on $V(\G')$ by letting  
$u \sim v$ whenever $|\G'(u) \cap \G'(v)|=4$ for any $u, v \in V(\G')$. 
Using that $m \ge 12$, it is not hard to show that $\sim$ is an equivalence relation with 
classes in the form 
$$
B_i:=\{u_i,u_{i+m},v_{i+1},v_{i+1+m}\},~i \in \Z_n.
$$
Clearly, $\sim$ is invariant under $\aut(\G')$, so the classes above form a block 
system for $\aut(\G')$. Part of the graph with $k=m-1$ is shown in Figure~1.

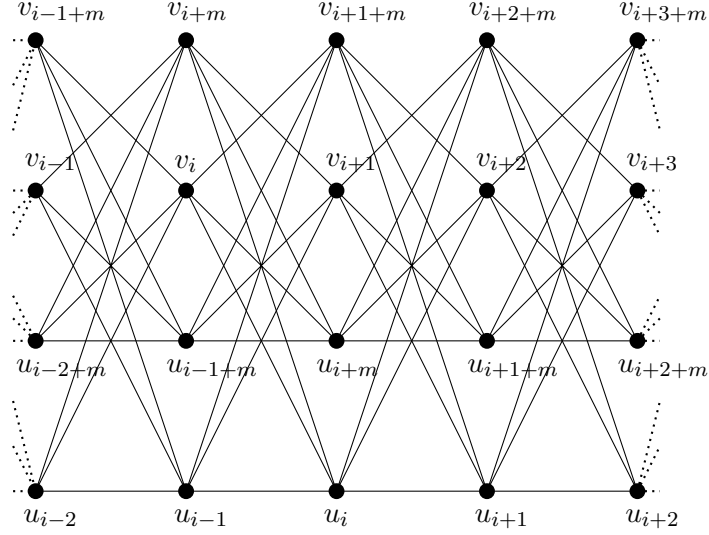
\begin{figure}[t!]
\begin{tikzpicture}[scale=1,line width=0.4pt]
\foreach \x in {0,2,4,6,8}
\foreach \y in {0,2,4,6} 
\fill (\x,\y) circle (3pt);
\foreach \x in {2,4,6,8}
\draw (\x,0) -- (\x-2,4) -- (\x,2) -- (\x-2,6) -- (\x,0);
\foreach \x in {0,2,4,6}
\draw (\x,0) -- (\x+2,4) -- (\x,2) -- (\x+2,6) -- (\x,0);
\draw (0,0) -- (8,0) (0,2) -- (8,2);
\draw (0,4) -- (2,6) -- (4,4) -- (6,6) -- (8,4);
\draw (0,6) -- (2,4) -- (4,6) -- (6,4) -- (8,6);
\draw 
(0.2,-0.1) node[below] {\small $u_{i-2}$} 
(2.2,-0.1) node[below] {\small $u_{i-1}$}
(4,-0.1) node[below] {\small $u_{i}$}
(6.2,-0.1) node[below] {\small $u_{i+1}$} 
(8.2,-0.1) node[below] {\small $u_{i+2}$}; 
\draw 
(0.35,1.9) node[below] {\small $u_{i-2+m}$} 
(2.35,1.9) node[below] {\small $u_{i-1+m}$}
(4.15,1.9) node[below] {\small $u_{i+m}$}
(6.35,1.9) node[below] {\small $u_{i+1+m}$} 
(8.35,1.9) node[below] {\small $u_{i+2+m}$}; 
\draw 
(0.2,4.1) node[above] {\small $v_{i-1}$} 
(2,4.1) node[above] {\small $v_{i}$}
(4.2,4.1) node[above] {\small $v_{i+1}$}
(6.2,4.1) node[above] {\small $v_{i+2}$} 
(8.2,4.1) node[above] {\small $v_{i+3}$}; 
\draw 
(0.35,6.1) node[above] {\small $v_{i-1+m}$} 
(2.15,6.1) node[above] {\small $v_{i+m}$}
(4.35,6.1) node[above] {\small $v_{i+1+m}$}
(6.35,6.1) node[above] {\small $v_{i+2+m}$} 
(8.35,6.1) node[above] {\small $v_{i+3+m}$}; 
\draw[dotted,thick] 
(-0.3,4.8) -- (0,6) (8,6) -- (8.3,4.8)
(-0.3,5.4) -- (0,6) (8,6) -- (8.3,5.4)
(-0.3,6) -- (0,6) (8,6) -- (8.3,6)
(-0.3,3.4) -- (0,4) (8,4) -- (8.3,3.4)
(-0.3,3.7) -- (0,4) (8,4) -- (8.3,3.7)
(-0.3,4) -- (0,4) (8,4) -- (8.3,4)
(-0.3,2.6) -- (0,2) (8,2) -- (8.3,2.6)
(-0.3,2.3) -- (0,2) (8,2) -- (8.3,2.3)
(-0.3,2) -- (0,2) (8,2) -- (8.3,2)
(-0.3,1.2) -- (0,0) (8,0) -- (8.3,1.2)
(-0.3,0.6) -- (0,0) (8,0) -- (8.3,0.6)
(-0.3,0) -- (0,0) (8,0) -- (8.3,0);
\end{tikzpicture}
\caption{The Nest graph $\nest(2m;2,m,2+m,m-1)$.}
\end{figure}

Let $K$ be the kernel of the action of $\aut(\G')$ on the latter block system. 
We prove next that $K$ is faithful on every block. Suppose that $g$ fixes pointwise 
the block $B_i$ for some $i \in \Z_n$. 
Any pair of vertices in $B_{i+1}$ are contained in a unique $4$-cycle intersecting 
$B_i$ at two vertices (see Figure~1). This means that $g$ maps any pair to itself, implying that $g$ fixes pointwise $B_{i+1}$. Repeating the argument, 
we conclude that $g$ fixes pointwise each block $B_i$, i.e., $g$ is 
the identity automorphism. 

Let $n=2m$, $C=\la \rho \ra$, where $\rho=(u_0,u_1,\ldots,u_{n-1})(v_0,v_1,\ldots,v_{n-1})$.
Using the facts that $K$ is faithful on every block $B_i$ and the quotient graph $\G/\B$ is an $n/2$-cycle whose automorphism group is isomorphic to 
the dihedral group $D_{n/2}$ of order $n$, we obtain the bound 
$$
|\aut(\G')| \le |K| \cdot n=24n.
$$ 
Thus $|C|^2=n^2 \ge |\aut(\G')|$ because $n=2m \ge 24$. By Theorem~\ref{L},  
$C$ has a non-trivial core in $\aut(\G')$, let this core be denoted by $N$.  

Assume first that $|N|$ is even. Then as $C_2$ is characteristic in $N$ and 
$N \lhd  \aut(\G')$, we obtain that $C_2 \lhd \aut(\G')$. Thus 
$\orb(C_2,V(\G))$ is a block system for $\aut(\G')$. But this is impossible because $u_0$ has 
one neighbour from the orbit $\{u_1,u_{1+m}\}$ and two from the orbit 
$\{v_0,v_m\}$. 

Let $|N|$ be odd and choose an odd prime divisor $p$ of $|N|$. It follows as 
above that $C_p \lhd \aut(\G')$. Clearly, $p$ divides $m$. 

Assume that $m=p$ or $2p$. If $m=p$, then by our initial assumptions, $k=p-1$. 
But this means that the edge $\{v_0,v_k\}$ is contained in a $C_p$-orbit, contradicting 
that $C_p \lhd \aut(\G')$ and $\G'$ is edge-transitive.
If $m=2p$, then $u_0$ has one neighbour from the $C_p$-orbit 
$\{ u_{4i+1} : 0 \le i \le p-1\}$ and two from the $C_p$-orbit 
$\{ v_{4i} : 0 \le i \le p-1\}$ (namely, $v_0$ and $v_{2+m}=v_{2+2p}$), which 
is a contradiction again.

Let $m > 2p$. By Lemma~\ref{KR1}(3) and the remark after the lemma, 
$$
\G'/C_p \cong \nest(2m/p;f(2),f(m),f(2+m),f(k)),
$$ 
where $f$ is the homomorphism from $\Z_{2m}$ to $\Z_{2m/p}$  such that $f(1)=1$. 
Since $m > 2p$ and $p$ is odd, it follows that 
$$
f(2)=2,~f(m)=m/p~\text{and}~f(2+m)=2+m/p.
$$ 
Furthermore, $f(k)=1$ if $k=1$ and $f(k)=m/p-1$ if $k=m-1$. 
By the minimality of $m$, we see that $m/p$ is odd and $f(k)=1$. 
This, however, contradicts that $m$ is even or $k=m-1$.  
\end{proof}

\begin{proof}[Proof of Proposition~\ref{family}]
Let $H$ and $N$ be the setwise and the pointwise stabiliser, 
respectively, of the set $\{ u_i : i \in \Z_n\}$ in $\aut(\G)$. 
Then $N \ne 1$ and $N \lhd H$. 
It follows that the $N$-orbits contained in $V:=\{ v_i : i \in \Z_n\}$ form  
a block system for the action of $H$ on $V$, implying that 
$\orb(N,V)=\orb(C_d,V)$ for some $d >1$.  
This yields that $n=2m$ and we may write w.l.o.g.~that 
$$
a < m,~b=m,~c=a+m,~\text{and}~k < m.
$$ 

Let $\eta$ be the permutation of the vertex set acting as  
$$
u_i^\eta=u_i~\text{and}~v_i^\eta=v_{i+m}~(i\in \Z_n).
$$   
It is easy to check that $\eta \in \aut(\G)$.

Note that, by Lemma~\ref{JMSV1}, $\G$ is arc-transitive, so 
$\aut(\G)_{u_0}$ is transitive on $\G(u_0)$. Let $s=|\G(v_0) \cap \G(v_m)|$. 
It is easy to see that $s \ge 4$. Define the graph $\Delta$ as follows: 
$$
V(\Delta)=\G(u_0)~\text{and}~E(\Delta)=\{ \{w,w'\} : |\G(w) \cap \G(w')|=s\}.
$$ 
Note that $\Delta$ is vertex-transitive, in particular, it is regular. 

Assume for the moment that $u_1$ and $u_{-1}$ are adjacent in $\Delta$. 
This means $|\G(u_1) \cap \G(u_{-1})|=s \ge 4$. Since 
$\G(u_1) \cap \G(u_{-1}) \cap \{u_i : i \in \Z_n\}=\{u_0\}$, we conclude that 
\begin{equation}\label{eq:cap}
|\{1, 1+a, 1+m, 1+a+m \} \cap \{-1,-1+a,-1+m,-1+a+m\}| \ge 3.
\end{equation}
At least one of $1$ and $1+m$ is in the intersection. If 
it is $1$, then $1=-1+a$ or $-1+a+m$ because $n >4$.  
As $a < m$, we find that $a=2$. 
Similarly, if $1+m$ is in the intersection, then $1+m=-1+a$ or $-1+a+m$, and 
we find again that $a=2$. Now, substituting $a=2$ in \eqref{eq:cap}, 
a contradiction arises because $n > 8$. 

Therefore, $u_1$ and $u_{-1}$ are not adjacent in $\Delta$. Using also that 
$\eta \in \aut(\G)$, see above, we obtain that $u_1$ must be 
adjacent with $v_0$ or $v_a$. 

Assume first that $u_1$ and $v_0$ are adjacent.  
Then $\G(u_1) \cap \G(v_0)=\{u_0,u_2,v_k,v_{-k}\}$, hence  
$2 \in \{0,n-a,m,n-a+m\}$ and $k \in \{1,1+a,1+m,1+a+m\}$.  
Since $a, k < m$ and $n > 4$, we find in turn that $a=m-2$, and $k=1$ or $m-1$. 

Now, if $u_1$ and $v_a$ are adjacent, then 
$\G(u_1) \cap \G(v_a)=\{u_0,u_2,v_{a+k},v_{a-k}\}$, whence  
$2 \in \{a,0,a+m,m\}$ and $a+k \in \{1,1+a,1+m,1+a+m\}$.  
Since $a, k < m$ and $n > 4$, we find in turn that $a=2$, and $k=1$ or $m-1$. 

To sum up, $a=2$ or $m-2$ and $k=1$ or $m-1$, and so the proposition 
follows from Lemma~\ref{help}. 
\end{proof}

\begin{cor}\label{cor-family}
Let $\G$ be a $G$-edge-transitive Nest graph of order $2n$ such that $n > 8$ and 
$$
C < G,~C=\la \rho \ra,~\text{and}~  
\rho=(u_0,u_1,\ldots,u_{n-1})(v_0,v_1,\ldots,v_{n-1}),
$$ 
and suppose that $\orb(C_{n/2},V(\G))$ is a block system for $G$. 
Then $C_{n/2} \lhd G$. 
\end{cor}
\begin{proof}
Let $K$ be the kernel of the action of $G$ on the block system $\orb(C_{n/2},V(\G))$, 
and let $K^*$ and $(C_{n/2})^*$ denote the image of 
$K$ and $C_{n/2}$, respectively, induced by their action on $U:=\{ u_i : i \in \Z_n\}$. 
Since the subgraph of $\G$ induced by $U$ is a cycle of length $n > 8$, it 
follows that $(C_{n/2})^*$ is characteristic in $K^*$. 
Therefore, if $K$ is faithful on $U$, then $C_{n/2}$ is characteristic in $K$ and 
as $K \lhd G$, we obtain that $C_{n/2} \lhd G$, as required.

If $K$ is not faithful on $U$, then by Proposition~\ref{family}, 
$n=2m$, $m$ is odd, and $\G \cong \G':=\nest(2m;2,m,2+m;1)$. 
Consider the group $\la C, \varphi, \eta \ra$, where 
$\varphi$ and $\eta$ are defined in Lemma~\ref{JMSV3}. 
This is transitive on $V(\G')$ and also contains the stabiliser of $u_0$ in 
$\aut(\G')$, therefore, $\aut(\G')=\la C, \varphi, \eta \ra$. 
A straightforward computation shows that $\varphi\rho^2\varphi=\rho^{-2}$ and 
$\eta \rho^2=\rho^2 \eta$, and hence $C_{n/2} \lhd \aut(\G')$. 
All these show that  $C_{n/2} \lhd G$ holds in this case as well. 
\end{proof}
\section{Proof Theorem~\ref{main}}\label{sec:proof}
Throughout this section we keep the following notation.

\begin{hypo}\label{hypo}\

\begin{quote}
$\G=\nest(n;a,b,c;k)$ is a Nest graph of order $2n$, $n \ge 4$, \\ [+1ex]
$C=\la \rho \ra$, where $\rho=(u_0,u_1,\ldots,u_{n-1})(v_0,v_1,\ldots,v_{n-1}),
$ \\ [+1ex]
$G \le \aut(\G)$ such that $\rho \in G$, $G$ acts transitively on $E(\G)$, 
and $\core_G(C)=1$.
\end{quote}
\end{hypo}

Instead of Theorem~\ref{main} we show the following slightly more stronger theorem. 
The proof will be given in the end of the section.

\begin{thm}\label{main+}
Assuming Hypothesis~\ref{hypo}, 
$\G$ is isomorphic to one of the graphs: $\nest(5;1,2,3;2)$, 
$\nest(8;1,3,4;3)$, $\nest(8;1,2,5;3)$ and $\nest(12;2,4,8;5)$.
\end{thm}

We start with a computational result, which we retrieved from 
\cite[Table~1]{JMSV} with the help of {\sc Magma}~\cite{BCP}. 
Here we use the obvious facts that  
$C$ is also core-free in $\aut(\G)$ and that $\aut(\G)$ is primitive 
whenever so is $G$.

\begin{lem}\label{small}
Assuming Hypothesis~\ref{hypo}, if $n \le 50$, then 
the following hold.
\begin{enumerate}[{\rm (1)}]
\item $\G$ is isomorphic to one of the graphs: 
$$
\nest(5;1,2,3;2),~\nest(8;1,3,4;3),~\nest(8;1,2,5;3)~\text{and}~ 
\nest(12;2,4,8;5).
$$
\item $G$ is either primitive and $\G \cong \nest(5;1,2,3;2)$ or $\nest(8;1,3,4;3)$; 
or for $N:=\soc(\aut(\G))$,    
$N \cong \Z_2^2$ if $n=8$ and $N \cong \Z_2^4$ if $n=12$, and the $N$-orbits have 
length $4$.  
\end{enumerate}
\end{lem}

The existence of a non-trivial non-cyclic block system is established next.

\begin{lem}\label{imp}
Assuming Hypothesis~\ref{hypo}, suppose that $n > 8$. 
Then $G$ admits a non-trivial non-cyclic block system.  
\end{lem}
\begin{proof}
Observe that, if $G$ is primitive, then Corollary~\ref{cor-DGJ} shows 
that $\G \cong \nest(5;1,2,3;2)$ or 
$\nest(8;1,3,4;3)$. As $n > 8$, $G$ is imprimitive.  

Let $\B$ be a non-trivial block system with blocks of size $d$. 
If $\B$ is cyclic, then by Lemma~\ref{KR1}(1) and Corollary~\ref{cor-family}, $C_d \lhd G$, where $C_d$ is the subgroup of $C$ of order $d$. This contradicts our assumption that $\core_G(C)=1$, so $\B$ is non-cyclic. 
\end{proof}

In the next two lemmas we study non-cyclic block systems with blocks of size $2$. 

\begin{lem}\label{2}
Assuming Hypothesis~\ref{hypo},  
suppose that $n > 50$ and $\B$ is a non-cyclic block system for $G$ with blocks of size $2$. Then $\G/\B$ has valency $12$.
\end{lem}
\begin{proof} 
Let $K$ be the kernel of the action of $G$ on $\B$, and for a subgroup $X \le G$, 
denote by $\bar{X}$ the image of $X$ induced by its action on $\B$. 
For a block $B \in \B$, we write $B=\{u_B,v_B\}$, where 
$u_B \in \{u_i : i \in \Z_n\}$ and by $v_B \in \{v_i : i \in \Z_n\}$, and define the permutation $\tau$ of $V(\G)$ as  
\begin{equation}\label{eq:tau}
\tau:=\prod_{B \in \B} (u_B\, v_B).
\end{equation}
Observe that $\tau$ commutes with any element of $G$. 

Now define the graph $\G'$ as 
\begin{equation}\label{eq:G'}
V(\G'):=V(\G)~\text{and}~E(\G'):=\{ \{u_0,u_1\}^x : x \in \la G, \tau \ra \}  
\end{equation}
Then $E(\G)=\{ \{u_0,u_1\}^x : x \in G\} \subseteq E(\G')$. 
Also, $\la \tau, G \ra \le \aut(\G')$, hence $\G'$ is both vertex- and edge-transitive. 
Since $\tau$ commutes with every element of $G$, it follows that 
$E(\G')=E(\G) \cup E(\G)^\tau$ and $E(\G)=E(\G)^\tau$ 
or $E(\G) \cap E(\G)^\tau=\emptyset$. Notice that   
$$
\G/\B=\G'/\B,
$$
hence we are done if show that $\G'/\B$ has valency $12$. 

Denote by $d$ and $d'$ the valency of $\G'$ and 
$\G'/\B$, respectively.
Now, $d=|E(\G')|/n=6 \frac{|E(\G')|}{|E(\G)|}$. This shows that  
$d=6$ if $E(\G)=E(\G')$, and $d=12$ otherwise. 

Assume for the moment that $d=6$, i.e., $E(\G)=E(\G')$. 
In this case $\tau \in K$, hence $\B$ is normal and $\G'$ is a normal $r$-cover 
of $\G'/\B$ and $r=1$ or $r=2$. 

If $r=2$, then $d'=3$. As $n > 50$, this is impossible due to Lemma~\ref{BFSX}.
Here we use the facts that $\G'/\B$ is edge-transitive and $\bar{C}$ is regular on 
$V(\G'/\B)$. It is well-known that an edge-transitive circulant graph is also 
arc-transitive. Thus $r=1$ and $d'=6$. 
As $n > 50$, Lemma~\ref{6} can be applied to $\G'/\B$ and 
$\bar{C}$. This says that $\aut(\G'/\B)$ has a normal subgroup $N$ such 
that 
\begin{enumerate}[{\rm (1)}]
\item $N=\bar{C}$, or 
\item $n \equiv 4 \!\!\pmod 8$ and $N=\bar{C}_{n/4}$, or   
\item $N \cong \Z_3^\ell$ for $\ell \ge 2$ and $\bar{C}_3 \le N$. 
\end{enumerate}

In case (1), $N < \bar{G}$, hence $KC \lhd G$. 
The condition $r=1$ yields that $K=\la \tau \ra$. Thus $KC$ is abelian and 
$\la x^2 : x \in KC\ra=C_n$ if $n$ is odd and $C_{n/2}$ if $n$ is even. 
Using that the latter group is characteristic in $KC$ and $KC \lhd G$, we obtain that 
$\core_G(C) \ne 1$, a contradiction.

In case (2), $N < \bar{G}$, hence $KC_{n/4} \lhd G$. 
Since $KC$ is abelian, it follows that $C_{n/4}$ is characteristic in $KC$, implying that 
$C_{n/4} \lhd G$,  a contradiction.

In case (3), $\bar{C}_3 \le \bar{G} \cap N$.  Since 
$\bar{G} \cap N \lhd \bar{G}$, it follows that $G$ contains 
a normal subgroup $M$ such that $M=\la \tau \ra \times S$, where 
$S \cong \Z_3^{\ell'}$ for some $\ell' \ge 1$. Thus $S$ is normal in $G$, and 
we obtain that $\orb(S,V(\G))$ is a non-trivial cyclic block system for $G$. 
This contradicts  Lemma~\ref{imp}, and we conclude that $d=12$. 

The graph $\G'$ is a normal $r$-cover of $\G'/\B$, where $r=1$ or $r=2$, and 
we have $d'=d/r=12/r$.  
If $r=2$, then $u_0$ is adjacent with $6$ vertices that are contained in the set 
$\{u_i : i \in \Z_n\}$. It follows 
from the definition of $\G'$ that this impossible. Thus $r=1$, and so $d'=12$.  
\end{proof}

\begin{lem}\label{2+}
Assuming Hypothesis~\ref{hypo},  
suppose that $n > 50$ and $\B$ is a non-cyclic block system for $G$ with blocks of size $2$. Then there is a normal non-cyclic block system for $G$ with blocks of size $4$. 
\end{lem}
\begin{proof}
Let $K$ be the kernel of the action of $G$ on $\B$, and for a subgroup $X \le G$, 
denote by $\bar{X}$ the image of $X$ induced by its action on $\B$. 
For the sake of simplicity we write $\bar{\G}$ for $\G/\B$.

By Lemma~\ref{2}, $\bar{\G}$ has valency $12$. 
This implies that $K=1$. 
As $\bar{C} \le \aut(\bar{\G})$ and it is regular on $V(\bar{\G})$, 
Theorem~\ref{K} can be applied to $\bar{\G}$ and $\bar{C}$. 
As $n > 50$, $\bar{\G}$ cannot be the complete graph. 
Also, if $\bar{C} \lhd 
\aut(\bar{\G})$, then $C \lhd G$ because $K=1$. 
This is also impossible, hence $\bar{\G}$ is in one of the families (c) and (d) of 
Theorem~\ref{K}.
\medskip

\noindent{\it Case 1.} $\bar{\G}$ is in family~(c). 
\medskip

In this case $\orb(\bar{C}_d,V(\bar{\G}))$ is a block system for $\aut(\bar{\G})$, hence for $\bar{G}$ as well, where $d \in \{2, 3, 4, 6\}$. Let $N$ be the unique 
subgroup of $G$ for which $\bar{N}$ is the kernel of the action of 
$\bar{G}$ on $\orb(\bar{C}_d,V(\bar{\G}))$. Note that $N \lhd G$ and  
$N \cong \bar{N}$ because $K=1$. 
Let $\B'=\orb(N,V(\G))$. It follows that $\B'$ is non-cyclic and it has blocks of size $2d$. 

Let $d=2$. Then $\B'$ is normal with blocks of size $4$, so the conclusion of the lemma holds. 

Let $d=3$. Then the Sylow $3$-subgroup of 
$\bar{N}$ is normal in $\bar{G}$. It follows in turn that, 
the Sylow $3$-subgroup of $N$ is normal in $G$, the orbits of the latter subgroup form 
a non-trivial cyclic block system for $G$. This contradicts Lemma~\ref{imp}.

Let $d=4$. Then $\bar{\G}$ has valency $3$. It follows from 
Lemma~\ref{BFSX} that $n \le 6$, but this is excluded. 

Finally, let $d=6$. Let $\tau$ be the permutation of $V(\G)$ defined in \eqref{eq:tau} and 
$\G'$ be the graph defined in \eqref{eq:G'}. 
Let $\Delta$ be the subgraph of $\G'$ induced by the set $u_0^N \cup u_1^N$. 
It is not hard to show that $\Delta$ is a bipartite graph, it has valency $6$, and 
it is also edge-transitive. Moreover, if $B \in \B$ such that $u_0 \notin B$ and 
$B \subset u_0^N \cup u_1^N$, then  
\begin{equation}\label{eq:cap}
| \Delta(u_0) \cap B|=1.
\end{equation}
Since $\la \tau \ra \times C_6 \le \aut(\Delta)$, it follows that 
$\Delta$ is uniquely determined by $\Delta(u_0)$.  It follows from the 
definition of $\G'$ that $|\G'(u_0) \cap \{u_i : i \in \Z_n\}|=4$. Therefore, replacing $u_0^N \cup u_1^N$ with $u_0^N \cup  u_{n-1}^N$ if necessary, we may assume w.l.o.g.~that $|\Delta(u_0) \cap \{u_i : i \in \Z_n\}| \le 2$. This together with 
\eqref{eq:cap} show that 
there are $6$ possibilities for $\Delta$. A computation with 
{\sc Magma}~\cite{BCP} shows that none of these $6$ graphs is edge-transitive. 
\medskip 

\noindent{\it Case 2.} $\bar{\G}$ is in family~(d). 
\medskip

We finish the proof by showing this case does not occur.
Theorem~\ref{K} shows that $\B_1:=\orb(\bar{C}_d,V(\bar{\G}))$ and 
$\B_2:=\orb(\bar{C}_{n/d},,V(\bar{\G}))$ are blocks for 
$\bar{G}$ for some divisor $d$ of $n$ such that 
$d \in \{4, 5, 7\}$ and $\gcd(d,n/d)=1$. Furthermore, 
$$
\bar{C}_d \times \bar{C}_{n/d} < 
\bar{G} \le \aut(\G')=G_1 \times G_2,
$$
where $\bar{C}_d \le G_1$, $G_1 \cong S_d$,   
$\bar{C}_{n/d} < G_2$ and $G_2 \cong \aut(\G'/\B_1)$.    

Let $d=4$. Then $\bar{\G}/\B_1$ has valency $4$. Using also that $n/d$ is odd and 
that $n > 20$, it follows 
from Lemma~\ref{BFSX} that $\bar{C}_{n/d} \lhd \bar{G}$, hence  
$C_{n/d} \lhd G$, a contradiction.

Let $d=5$. Then $\bar{\G}/\B_1$ has valency $3$, hence $n \le 30$ 
by Lemma~\ref{BFSX}, which is excluded.

Finally, let $d=7$. Then $\bar{\G}/\B_1$ is a cycle of length $n/7$, implying that 
$\bar{C}_{n/d} \lhd \bar{G}$, so $C_{n/d} \lhd G$, a contradiction. 
\end{proof}

Before the proof of Theorem~\ref{main+} we need two more lemmas dealing with 
non-cyclic block systems with blocks of size at least $4$.

\begin{lem}\label{>2}
Assuming Hypothesis~\ref{hypo}, suppose that $n > 50$ and  
$\B$ is a minimal non-cyclic block system for $G$ with blocks of size at least $4$, and 
let $B \in \B$ be any block. Then the permutation group of $B$ induced by 
$G_{\{B\}}$ is an affine group.
\end{lem}
\begin{proof}
For a subgroup $X \le G_{\{B\}}$, denote by $X^*$ the image of  
$X$ induced by its action on $B$. 
As $B$ is minimal, $(G_{\{B\}})^*$ is a primitive permutation group. Also, 
$(C_{\{B\}})^*$ is a semiregular cyclic subgroup of $(G_{\{B\}})^*$ with $2$ orbits, hence Theorem~\ref{M} can be applied to $(G_{\{B\}})^*$. 
This shows that $(G_{\{B\}})^*$ is either an affine group or it is one of the groups in the families (a)-(f) in part (2) of Theorem~\ref{M}. 
Assume that the latter case occurs. We drive in three steps that this 
leads to a contradiction. Let $K$ be the kernel of the action of $G$ on $\B$. 
\medskip

\noindent
{\it Step~1. $K$ acts faithfully on every block in $\B$.} 
\medskip

Since $K \lhd G_{\{B\}}$, it follows that 
$K^* \lhd (G_{\{B\}})^*$. By Corollary~\ref{M-p1}, $K^*$ is primitive and belongs to 
the same family as $(G_{\{B\}})^*$.  

Assume on the contrary that $K$ is not faithful on every block.
Using the connectedness of $\G$, it is easy to show that there are blocks $B, B'$ 
in $\B$ with the following properties: The kernel of the action of $K$ on $B$ is non-trivial on $B'$, and $\G$ has an edge $\{w,w'\}$ such that $w \in B$ and $w' \in B'$. Denote by $N$ the latter kernel. Now as $N \lhd K$ and 
$K$ is primitive on $B'$, $N$ is transitive on $B'$. Thus the orbit $(w')^N=B'$, 
and so $w$ is adjacent with any vertex in $B'$. 
Since $\B$ is normal, it follows that the subgraph of $\G$ induced by 
$B \cup B'$ is isomorphic to the complete bipartite graph $K_{m,m}$, where 
$m=|B|$. On the other hand, $m \ge 6$, showing that $\G \cong K_{6,6}$, a 
contradiction. 
\medskip

Denote by $B$ and $B'$ the blocks containing $u_0$ and $u_1$, respectively. 
\medskip

\noindent
{\it Step~2.} The action of $K$ on $B$ is equivalent with its action on $B'$. 
\medskip

Assume on the contrary that the actions are inequivalent. 
Due to Lemma~\ref{M-p2&3}(2), $K$ belongs to family (c) in  
Theorem~\ref{M}(2) with $d \ge 4$, and the elements in $B$ and $B'$ 
correspond to the points and the hyperplanes of the projective geometry 
$\PG(d-1,q)$, respectively. The set $B'$ splits into two $K_{u_0}$-orbits 
of lengths  
$$
(q^{d-1}-1)/(q-1)~\text{and}~q(q^{d-1}-1)/(q-1).
$$  
The first orbit consists of the hyperplanes of $\PG(d-1,q)$ 
through the point represented by $u_0$,   
and the second orbit consists of the remaining hyperplanes. 
Clearly, the minimum of these numbers is bounded above by the valency of $\G$, 
implying that $q^{d-1}-1 \le 6(q-1)$, and hence $q^{d-2} < 6$. This is impossible because $d \ge 4$. 
\medskip

\noindent
{\it Step~3.} $\core_G(C) \ne 1$. 
\medskip
  
Since $K$ acts equivalently on $B$ and $B'$, it follows that $K_{u_0}=K_v$ for some vertex $v \in B'$ (see \cite[Lemma~1.6B]{DM}). 
Define the binary relation $\sim$ on $V(\G)$ by letting $u \sim v$ if and only if 
$K_u=K_v$. It is not hard to show, using that $K \lhd G$, that $\sim$ is 
a $G$-congruence (see \cite[Exercise~1.5.4]{DM}), and so there is a block for $G$ containing $u_0$ and $v$. Also, as $K$ is not regular, this block is non-trivial, and this shows that $v \ne u_1$.  

By Lemma~\ref{M-p2&3}(1), $K$ is $2$-transitive on $B'$, unless  
$|B'|=10$, $K=A_5$ or $S_5$, and it has subdegrees $1, 3$ and $6$. 

Assume first that $K$ is $2$-transitive on $B'$.
Then the orbit $u_1^{K_{u_0}}=u_1^{K_v}=|B'|-1$ and each vertex in $u_1^{K_v}$ is adjacent with $u_0$. Hence $u_0$ has $|B'|-1$ neighbours in $B'$. On the other hand, as $\B$ is normal, this number divides $6$, so $|B|=|B'|=4$, contradicting that 
$G_{\{B\}}^*$ is an almost simple group.

We are left with the case that $|B'|=10$,  
$K=A_5$ or $S_5$, and it has subdegrees $1, 3$ and $6$. 
Consequently, $u_0$ has $3$ or $6$ neighbours in $B'$. 

If $u_0$ has $6$ neighbours, then it is clear that $n=10$, which is excluded. 

Now assume that $u_0$ has $3$ neighbours in $B'$. 
In this case $\G$ is a normal $3$-cover of a cycle of length $n/5$. 
Since $\G/\B$ is a cycle of length $n/5$, it follows that 
$|G| \le |K| \cdot 2n/5 =48 n$. Using that $n > 50$, 
Theorem~\ref{L} shows that $\core_G(C) \ne 1$.
\end{proof}

\begin{lem}\label{>2+}
Assuming Hypothesis~\ref{hypo}, suppose that $n > 50$ and 
$\B$ is a minimal non-cyclic block system for $G$ with blocks of size at least $4$.  
Then the blocks have size $4$.
\end{lem}
\begin{proof}
Let $K$ be the kernel of the action of $G$ on $\B$, and let $B \in \B$ be the block containing $u_0$. Denote by $(G_{\{B\}})^*$ the permutation group of $B$ induced 
by $G_{\{B\}}$. By Lemma~\ref{>2}, $(G_{\{B\}})^*$ is an affine group, and thus it is 
one of the groups in the families (a)--(d) in part (1) of Theorem~\ref{M}. 
In particular, $|B| \in \{4, 8, 16 \}$. 
Assume on the contrary that $|B| > 4$. 
By Lemma~\ref{KR2}, $\B$ is normal, hence 
$\G$ is a normal $r$-cover of $\G/\B$ for $r \in \{1,2,3\}$. 
\medskip

\noindent{\it Case~1.} $r=1$.
\medskip

In this case $K$ is regular on every block, in particular, 
$K \cong \Z_2^4$ or $\Z_2^5$. On the other hand, by Lemma~\ref{KR2}(2),  
$C_{|B|/2} < K$, a contradiction. 
\medskip

\noindent{\it Case~2.} $r=2$.
\medskip

In this case $\G/\B$ has valency $3$. 
It follows from Lemma~\ref{BFSX} that $n \le 48$, but this is excluded.
\medskip

\noindent{\it Case~3.} $r=3$.
\medskip

Then $\G/\B$ is a cycle of length $2n/|B|$. This implies that the action of 
$G_{u_0}$ on $\G(u_0)$  admits a block system consisting of two blocks of size $3$. 
Consequently, the restriction of $G_{u_0}$ to $\G(u_0)$ is a $\{2,3\}$-group. 
This together with the fact that $\G$ is connected yield that 
$G_{u_0}$ is also a $\{2,3\}$-group. Now checking  
the stabilisers in part (1) of Theorem~\ref{M}, we find that $|B|=16$ and 
\begin{equation}\label{eq:stab}
(G_{\{B\}})^*_{u_0} \cong (\Z_3 \times \Z_3) \rtimes \Z_4~\text{or}~
(S_3 \times S_3) \rtimes \Z_2 
\end{equation}
 
Assume for the moment $K$ is not faithful on $B$.  
Then there exist adjacent blocks $B'$ and $B''$ such that the apkernel 
of the action of $K$ on $B'$ is non-trivial on $B''$. Denote this kernel by $L$. 
The $L$-orbits contained in $B''$ have the same size, which is equal to $2^s$ for 
some $1 \le s \le 4$. On the other hand, for $w \in B'$, the set $\G(w) \cap B''$ 
is $L$-invariant, implying that $|\G(w) \cap B''|$ is equal to some power of $2$, a 
contradiction. Thus $K$ is faithful on $B$. 

The group $K$ contains a normal subgroup $E$ 
such that $E \cong \Z_2^4$. Note that $\B=\orb(E,V(\G))$. 
Let $P$ be the Sylow $3$-subgroup of $G_{\{B\}}$. 
Since $\G/\B$ is a cycle, it follows that $P \le K$. This also shows that $P \cong \Z_3^2$. 
Also, $C_8 \le K$, and in view of \eqref{eq:stab}, we obtain that 
$|(G_{\{B\}})^* : K| \le 2$ and if the index is equal to $2$, then 
$(G_{\{B\}})^*_{u_0} \cong (S_3 \times S_3) \rtimes \Z_2$. 
A direct check by {\sc Magma}~\cite{BCP} shows that in the latter case 
$(G_{\{B\}})^*$ has a unique subgroup of index $2$ containing an element of order 
$8$, which is also primitive. All these show that $K$ is primitive on $B$.

Denote by $\Delta$ be the subgraph of 
$\G$ induced by $u_0^E \cup u_1^E$. Using that $E \cong \Z_2^4$ acting 
regularly on both $u_0^E$ and $u_1^E$, it is not hard to show that 
$\Delta$ is the union of four 
$3$-dimensional cube $Q_3$. If $\Delta_1$ is a component of $\Delta$, then 
$|V(\Delta_1) \cap B|=4$ (note that $B=u_0^E$) and 
$V(\Delta_1) \cap B$ is a block for $K$.  This, however, contradicts the fact that $K$ is primitive on $B$.
\end{proof}

We are ready to settle Theorem~\ref{main+}, and therefore, 
Theorem~\ref{main} as well. 

\begin{proof}[Proof of Theorem~\ref{main+}.] 
In view of Lemma~\ref{small}, we may assume that $n > 50$. 
It follows from Lemmas~\ref{imp}--\ref{>2+} that $G$ admits a normal non-cyclic 
block system with blocks of size $4$. Denote this block system by $\B$.  
Let $K$ be the kernel of the action of $G$ on $\B$, and for a subgroup $X \le G$, 
denote by $\bar{X}$ the image of $X$ induced by its action on $\B$. 
As $\B$ is normal, $\G$ is a normal $r$-cover of 
$\G/\B$ for some $r \in \{1,2,3\}$. We exclude below all possibilities case-by-case.
\medskip

\noindent{\it Case~1.} $r=1$.
\medskip

In this case $|K|=4$ and $K \cap C=C_2$. 
If $K \cong \Z_4$, then $C_2$ is characteristic in $K$, and therefore, it is 
normal in $G$. This is impossible because $\core_G(C)=1$, hence 
$K \cong \Z_2^2$

The graph $\G/\B$ is edge-transitive, it has valency $6$, and 
$\bar{C}$ is regular on $V(\G/\B)$.  As $n > 50$, Lemma~\ref{6} can be applied to 
$\G/\B$ and $\bar{C}$. 
It follows that $\aut(\G/\B)$ has a normal subgroup $N$ such that 
\begin{enumerate}[{\rm (1)}]
\item $N=\bar{C}$, or 
\item $n \equiv 4 \!\!\pmod 8$ and $N=\bar{C}_{n/4}$, or 
\item $N \cong \Z_3^{\ell}$ for $\ell \ge 2$ and $\bar{C}_3 \le N$.
\end{enumerate}

In case (1), we obtain that $KC \lhd G$, whereas in case (2), $KC_{n/4} \lhd G$. 
In either case, $|KC:C|=2$, and therefore, for the derived subgroup 
$(KC)^\prime$, $(KC)^\prime \le C$. Thus $(KC)^\prime \le \core_G(C)$, and so 
$(KC)^\prime=1$, i.e., $KC$ is an abelian group. 
Then we obtain that $C_{n/2} \lhd G$ in case (1), 
and $C_{n/4} \lhd G$ in case (2). None of these is possible because $\core_G(C)=1$. 

In case (3), $G$ contains a normal subgroup $M$ such that $KC_3 \le M$ and 
$M/K \cong \Z_3^{\ell'}$ for some $\ell' \ge 1$. Then $M$ can be written as $M=KS$ where $C_3 \le S$ and $S \cong \Z_3^{\ell'}$. 
As $C_2 \le K$, we obtain that $C_3$ commutes with $K$, and so $C_3 \le O_3(M)$, 
where $O_3(M)$ denotes the largest normal $3$-subgroup of $M$.  As $O_3(M)$ is characteristic in $M$, $O_3(M) \lhd G$. This yields that $\orb(O_3(M),V(\G))$ is a non-trivial cyclic block system, a contradiction to Lemma~\ref{imp}. 
\medskip

\noindent{\it Case~2.} $r=2$.
\medskip

In this case $\G/\B$ has valency $3$, hence $n \le 12$ by Lemma~\ref{BFSX}, which is excluded.
\medskip

\noindent{\it Case~3.} $r=3$.
\medskip

The group $K$ is faithful on every block of $\B$. This can be shown by copying 
the argument that has been used in Case~3 in the proof of Lemma~\ref{>2+}.  
Since $\G/\B$ is a cycle of length $n/2$, it follows that 
$|G| \le |K| \cdot n =24 n$. 
Using that $n > 50$, Theorem~\ref{L} shows that $\core_G(C) \ne 1$, a 
contradiction.
\end{proof}

\end{document}